\documentclass[10pt,journal]{IEEEtran}
\usepackage[ansinew]{inputenc}
\usepackage[dvips]{graphicx}
\usepackage[T1]{fontenc}
\usepackage{amsmath,enumerate,amssymb,hhline,verbatim}

\begin{document}

\title{Generalization of Pigeon Hole Bound}
\author{\IEEEauthorblockN{Toni Ernvall}\\
\IEEEauthorblockA{Department of Mathematics,\\
University of Turku
}
}

\maketitle

\newtheorem{definition}{Definition}[section]
\newtheorem{thm}{Theorem}[section]
\newtheorem{proposition}[thm]{Proposition}
\newtheorem{lemma}[thm]{Lemma}
\newtheorem{corollary}[thm]{Corollary}
\newtheorem{exam}{Example}[section]
\newtheorem{conj}{Conjecture}
\newtheorem{remark}{Remark}[section]

\newcommand{\La}{\mathbf{L}}
\newcommand{\h}{{\mathbf h}}
\newcommand{\Z}{{\mathbf Z}}
\newcommand{\R}{{\mathbf R}}
\newcommand{\C}{{\mathbf C}}
\newcommand{\D}{{\mathcal D}}
\newcommand{\F}{{\mathbf F}}
\newcommand{\HH}{{\mathbf H}}
\newcommand{\OO}{{\mathcal O}}
\newcommand{\G}{{\mathcal G}}
\newcommand{\A}{{\mathcal A}}
\newcommand{\B}{{\mathcal B}}
\newcommand{\I}{{\mathcal I}}
\newcommand{\E}{{\mathcal E}}
\newcommand{\PP}{{\mathcal P}}
\newcommand{\Q}{{\mathbf Q}}
\newcommand{\M}{{\mathcal M}}
\newcommand{\separ}{\,\vert\,}
\newcommand{\abs}[1]{\vert #1 \vert}

\begin{abstract}
H.-F. Lu, J. Lahtonen, R. Vehkalahti, and C. Hollanti introduced so called Pigeon Hole Bound for decay function of MIMO-MAC codes. Here we give a generalization for it.
\end{abstract}

\section{Decay function}

We consider here decay function of MIMO-MAC codes. Every user has some specific lattice $\textbf{L}_j \subseteq \M_{n \times k}, j=1, \ldots, U$ with $k \geq Un$. We assume that each user's lattice is of full rank $r=2kn$ so the lattice $\textbf{L}_j$ has an integral basis $B_{j,1}, \ldots, B_{j,r}$. Now the code associated with $j$th user is a restriction of lattice $\mathbf{L}_{j}$ such that
$$
\textbf{L}_{j} (N_{j}) = \{ \sum_{i=1}^{r} b_{i} B_{j,i} | b_{i} \in \Z, -N_{j} \leq b_{i} \leq N_{j} \}
$$
where $N_{j}$ is a given positive number.

Using these definitions the $U$-user mimo-mac code is $(\textbf{L}_{1} (N_{1}), \textbf{L}_{2} (N_{2}), \dots, \textbf{L}_{U} (N_{U}))$.

For this we define
$$
\mathfrak{D}(N_1, \dots, N_U) = \min_{X_{j} \in \textbf{L}_{j}(N_{j}) \setminus \{ 0 \} }  \det (MM^{\dag})
$$
where $M=M(X_{1}, \ldots, X_{U})$. For a special case $N_1=\dots=N_U=N$ we write
$$
\mathfrak{D}(N) = \mathfrak{D}(N_1=N, \dots, N_U=N).
$$

In the special case $k=Un$ we have
$$
\mathfrak{D}(N_1, \dots, N_U) = D(N_{1}, \ldots, N_{U})^2
$$
and especially
$$
\mathfrak{D}(N) = D(N)^2.
$$

\section{An upper bound using pigeon hole principle}

So called Pigeon Hole Bound for decay function of MIMO-MAC codes was introduced in \cite{remarks}. Here we give a generalization for it.

\begin{lemma}
\label{matriisitulo}
Let $\mathbf{c}_1, \mathbf{c}_2, \dots, \mathbf{c}_k, \mathbf{e}_1, \mathbf{e}_2, \dots, \mathbf{e}_{k-1} \in \C^{n}$, and $\mathbf{c}_i-\mathbf{e}_i \in L(\mathbf{c}_{i+1}, \mathbf{c}_{i+2}, \dots, \mathbf{c}_{k})$ for $i=1, \dots, k-1$. Write also
$$
A = \left(
         \begin{array}{c}
           \mathbf{c}_1          \\
           \mathbf{c}_2          \\
           \vdots       \\
           \mathbf{c}_{k-1}          \\
           \mathbf{c}_k          \\
         \end{array}
       \right)
$$
and
$$
B = \left(
         \begin{array}{c}
           \mathbf{e}_1          \\
           \mathbf{e}_2          \\
           \vdots       \\
           \mathbf{e}_{k-1}      \\
           \mathbf{c}_k          \\
         \end{array}
       \right).
$$
Then we have $\det(AA^{\dag})=\det(BB^{\dag})$.
\end{lemma}
\begin{proof}
If $k>n$ then $\det(AA^{\dag})=0=\det(BB^{\dag})$. If $k=n$ then $\det(A)$ is
$$
\left|
         \begin{array}{c}
           \mathbf{c}_1          \\
           \mathbf{c}_2          \\
           \vdots       \\
           \mathbf{c}_{k-1}          \\
           \mathbf{c}_k          \\
         \end{array}
       \right|
 = \left|
         \begin{array}{c}
           \mathbf{e}_1          \\
           \mathbf{c}_2          \\
           \vdots       \\
           \mathbf{c}_{k-1}          \\
           \mathbf{c}_k          \\
         \end{array}
       \right|
 = \left|
         \begin{array}{c}
           \mathbf{e}_1          \\
           \mathbf{e}_2          \\
           \vdots       \\
           \mathbf{c}_{k-1}          \\
           \mathbf{c}_k          \\
         \end{array}
       \right|
=\dots
 = \left|
         \begin{array}{c}
           \mathbf{e}_1          \\
           \mathbf{e}_2          \\
           \vdots       \\
           \mathbf{e}_{k-1}          \\
           \mathbf{c}_k          \\
         \end{array}
       \right|
$$
i.e. $\det(B)$ and hence $\det(AA^{\dag})=\det(BB^{\dag})$.

Assume $k<n$. Let $\mathbf{v}_1, \dots, \mathbf{v}_{n-k} \in \C^{n}$ be such that $\mathbf{v}_1 \in L(\mathbf{c}_1, \mathbf{c}_2, \dots, \mathbf{c}_k,)^{\bot}\setminus \{\mathbf{0}\}$, $\mathbf{v}_2 \in L(\mathbf{v}_1, \mathbf{c}_1, \mathbf{c}_2, \dots, \mathbf{c}_k)^{\bot}\setminus \{\mathbf{0}\}$, ..., $\mathbf{v}_{n-k} \in L(\mathbf{v}_1, \mathbf{v}_2, \dots \mathbf{v}_{n-k-1}, \mathbf{c}_1 , \mathbf{c}_2, \dots, \mathbf{c}_k)^{\bot}\setminus \{\mathbf{0}\}$. Now (as in the case $n=k$) we have
$$
\det(\left(
         \begin{array}{c}
           \mathbf{c}_1          \\
           \vdots       \\
           \mathbf{c}_{k-1}      \\
           \mathbf{c}_k          \\
           \mathbf{v}_1          \\
           \vdots       \\
           \mathbf{v}_{n-k}      \\
           \end{array}
       \right))
=
\det(\left(
         \begin{array}{c}
           \mathbf{e}_1     \\
           \vdots       \\
           \mathbf{e}_{k-1}          \\
           \mathbf{c}_k          \\
           \mathbf{v}_1          \\
           \vdots       \\
           \mathbf{v}_{n-k}      \\
           \end{array}
       \right))
$$
and hence
$$
\left|
         \begin{array}{cccccc}
           \mathbf{c}_1 \mathbf{c}_{1}^{*}          & \dots   & \mathbf{c}_1 \mathbf{c}_{k}^{*}     & \mathbf{c}_1 \mathbf{v}_{1}^{*}     & \dots & \mathbf{c}_1 \mathbf{v}_{n-k}^{*}       \\
           \vdots                 &         & \vdots            & \vdots            &       & \vdots  \\
           \mathbf{c}_k \mathbf{c}_{1}^{*}          & \dots   & \mathbf{c}_k \mathbf{c}_{k}^{*}     & \mathbf{c}_k \mathbf{v}_{1}^{*}     & \dots & \mathbf{c}_k \mathbf{v}_{n-k}^{*}  \\
           \mathbf{v}_1 \mathbf{c}_{1}^{*}          & \dots   & \mathbf{v}_1 \mathbf{c}_{k}^{*}     & \mathbf{v}_1 \mathbf{v}_{1}^{*}     & \dots & \mathbf{v}_1 \mathbf{v}_{n-k}^{*}  \\
           \vdots                 &         & \vdots            & \vdots            &       & \vdots           \\
           \mathbf{v}_{n-k} \mathbf{c}_{1}^{*}      & \dots   & \mathbf{v}_{n-k} \mathbf{c}_{k}^{*} & \mathbf{v}_{n-k} \mathbf{v}_{1}^{*} & \dots & \mathbf{v}_{n-k} \mathbf{v}_{n-k}^{*}  \\
           \end{array}
       \right|
$$
is equal than
$$
\left|
         \begin{array}{cccccc}
           \mathbf{e}_1 \mathbf{e}_{1}^{*}          & \dots   & \mathbf{e}_1 \mathbf{c}_{k}^{*}     & \mathbf{e}_1 \mathbf{v}_{1}^{*}     & \dots & \mathbf{e}_1 \mathbf{v}_{n-k}^{*}       \\
           \vdots                 &         & \vdots            & \vdots            &       & \vdots  \\
           \mathbf{c}_k \mathbf{e}_{1}^{*}          & \dots   & \mathbf{c}_k \mathbf{c}_{k}^{*}     & \mathbf{c}_k \mathbf{v}_{1}^{*}     & \dots & \mathbf{c}_k \mathbf{v}_{n-k}^{*}  \\
           \mathbf{v}_1 \mathbf{e}_{1}^{*}          & \dots   & \mathbf{v}_1 \mathbf{c}_{k}^{*}     & \mathbf{v}_1 \mathbf{v}_{1}^{*}     & \dots & \mathbf{v}_1 \mathbf{v}_{n-k}^{*}  \\
           \vdots                 &         & \vdots            & \vdots            &       & \vdots           \\
           \mathbf{v}_{n-k} \mathbf{c}_{1}^{*}      & \dots   & \mathbf{v}_{n-k} \mathbf{c}_{k}^{*} & \mathbf{v}_{n-k} \mathbf{v}_{1}^{*} & \dots & \mathbf{v}_{n-k} \mathbf{v}_{n-k}^{*}  \\
           \end{array}
       \right|.
$$
And since the way we chose $\mathbf{v}_1, \dots, \mathbf{v}_{n-k}$ this means that
$$
\left|
         \begin{array}{cccccc}
           \mathbf{c}_1 \mathbf{c}_{1}^{*}          & \dots       & \mathbf{c}_1 \mathbf{c}_{k}^{*}     & 0     & \dots & 0       \\
           \vdots                 &                      & \vdots            & \vdots            &       & \vdots  \\
           \mathbf{c}_k \mathbf{c}_{1}^{*}          & \dots       & \mathbf{c}_k \mathbf{c}_{k}^{*}     & 0     & \dots & 0  \\
           0          & \dots       & 0     & \mathbf{v}_1 \mathbf{v}_{1}^{*}     & \dots &0  \\
           \vdots                 &                     & \vdots            & \vdots            &       & \vdots           \\
           0      & \dots  & 0 & 0 & \dots & \mathbf{v}_{n-k} \mathbf{v}_{n-k}^{*}  \\
           \end{array}
       \right|
$$
is equal than
$$
\left|
         \begin{array}{cccccc}
           \mathbf{e}_1 \mathbf{e}_{1}^{*}          & \dots       & \mathbf{e}_1 \mathbf{c}_{k}^{*}     & 0     & \dots & 0       \\
           \vdots                 &                     & \vdots            & \vdots            &       & \vdots  \\
           \mathbf{c}_k \mathbf{e}_{1}^{*}          & \dots      & \mathbf{c}_k \mathbf{c}_{k}^{*}     & 0     & \dots & 0  \\
           0          & \dots  & 0         & \mathbf{v}_1 \mathbf{v}_{1}^{*}     & \dots & 0  \\
           \vdots                 &                  & \vdots            & \vdots            &       & \vdots           \\
           0      & \dots  & 0 & 0 & \dots & \mathbf{v}_{n-k} \mathbf{v}_{n-k}^{*}  \\
           \end{array}
       \right|
$$
because if we write $\mathbf{e}_i = \mathbf{c}_i - \mathbf{x}_i$ where $\mathbf{x}_i \in L(\mathbf{c}_{i+1}, \dots, \mathbf{c}_k)$ then $\mathbf{v}_j \mathbf{e}_{i}^{*} = \mathbf{v}_j (\mathbf{c}_i - \mathbf{x}_i)^{*} = \mathbf{v}_j \mathbf{c}_{i}^{*} - \mathbf{v}_j \mathbf{x}_{i}^{*} = 0 - 0 = 0$ for all $i=1, \dots, n-1$ and $j=1, \dots, n-k$.
This gives that
$$
|\mathbf{v}_1|^2 \dots |\mathbf{v}_{n-k}|^2 \det(AA^{\dag}) = |\mathbf{v}_1|^2 \dots |\mathbf{v}_{n-k}|^2 \det(BB^{\dag})
$$
and hence $ \det(AA^{\dag}) = \det(BB^{\dag})$.
\end{proof}

\begin{lemma}
\label{projection}
Let $V=\R^n$ be an $n$-dimensional vector space, $N$ a given positive integer, and let $\mathbf{c}_1, \mathbf{c}_2, \dots, \mathbf{c}_n \in V$ be a basis for $V$. Let also $U$ be a $k$-dimensional subspace of $V$ and $\pi : V \rightarrow U$ an orthogonal projection into $U$. We can choose a basis $\{ \pi(\mathbf{c}_{j_1}) , \pi(\mathbf{c}_{j_2}) , \dots , \pi(\mathbf{c}_{j_k}) \}$ for $U$ such that if $\mathbf{v} = a_1 \mathbf{c}_1 + a_2 \mathbf{c}_2 + \dots + a_n \mathbf{c}_n$ with $|a_i| \leq N$ for all $i$ then $\pi(\mathbf{v}) = b_1 \pi(\mathbf{c}_{j_1}) + b_2 \pi(\mathbf{c}_{j_2}) + \dots + b_k \pi(\mathbf{c}_{j_k})$ with $|b_i| \leq n^{2}N$ for all $i$.
\end{lemma}
\begin{proof}
Since $\{ \mathbf{c}_1, \mathbf{c}_2, \dots, \mathbf{c}_n \}$ is a basis for $V$ we can choose a basis for $U$ from the set $\{ \pi(\mathbf{c}_{1}) , \pi(\mathbf{c}_{2}) , \dots , \pi(\mathbf{c}_{n}) \}$. Without loss of generality we may assume that $\pi(\mathbf{c}_{1}) , \pi(\mathbf{c}_{2}) , \dots , \pi(\mathbf{c}_{k})$ are linearly independent. We say that they form a basis $K_1$.

Let $\pi(c_{k+1}) = d_1 \pi(\mathbf{c}_{1}) + d_2 \pi(\mathbf{c}_{2}) + \dots + d_k \pi(\mathbf{c}_{k})$ for some $d_1, d_2, \dots, d_k \in \R$.

If $|d_i| \leq 1$ for all $i=1,2, \dots, k$ then let $K_2 = K_1$.

Otherwise let $|d_j| > 1$ be a maximal coefficient. Now $\pi(\mathbf{c}_{j}) = \frac{-d_1}{d_j} \pi(\mathbf{c}_{1}) + \frac{-d_2}{d_j} \pi(\mathbf{c}_{2}) + \dots + \frac{-d_{j-1}}{d_j} \pi(\mathbf{c}_{j-1}) + \frac{d_{k+1}}{d_j} \pi(\mathbf{c}_{k+1}) + \frac{-d_{j+1}}{d_j} \pi(\mathbf{c}_{j+1}) + \dots + \frac{-d_{k+1}}{d_j} \pi(\mathbf{c}_{k})$ and the absolute values of coefficients on the right hand side are smaller or equal that one. In this case let $K_2 = (K_1 \setminus {\pi(\mathbf{c}_j)}) \cup \{ \pi(\mathbf{c}_{k+1}) \}$. It is clear that $K_2$ is a basis for $U$.

Now similarly form a new basis $K_{i+1}$ using basis $K_i$ for all $i=1,2, \dots, n-k$ and write $K=K_{n-k+1}=\{ \pi(\mathbf{c}_{j_1}) , \pi(\mathbf{c}_{j_2}) , \dots , \pi(\mathbf{c}_{j_k}) \}$.

Now $\pi(\mathbf{c}_l) = d_{l,1} \pi(\mathbf{c}_{j_1}) + d_{l,2} \pi(\mathbf{c}_{j_2}) + \dots + d_{l,k} \pi(\mathbf{c}_{j_k})$ where $|d_{l,i}| \leq n$ for all $l$ and $i$ because every time when we took some $\pi(\mathbf{c}_h)$ off from the basis it then had a such representation in the new basis that all the coordinates were absolutely smaller or equal than zero. Repeating this procedure at most $n-k \leq n$ times and using triangle inequality gives then the property. Hence if $\mathbf{v} = a_1 \mathbf{c}_1 + a_2 \mathbf{c}_2 + \dots + a_n \mathbf{c}_n$ with $|a_i| \leq N$ then $\pi(\mathbf{v}) = b_1 \pi(\mathbf{c}_{j_1}) + b_2 \pi(\mathbf{c}_{j_2}) + \dots + b_k \pi(\mathbf{c}_{j_k})$ with $|b_i| \leq n^{2}N$ for all $i$.
\end{proof}

\begin{thm}
For a MIMO-MAC lattice code $(\textbf{L}_{1} (N_{1}), \textbf{L}_{2} (N_{2}), \dots, \textbf{L}_{U} (N_{U}))$ of $U$ users, each transmitting with $n$ transmission antennas, having a code of length $k \geq Un$, and each users lattice is of full rank $r=2kn$ we have a constant $K$ such that
$$
\mathfrak{D}(N_1, \dots, N_U) \leq K \prod_{l=1}^{U-1} N_{l}^{-\frac{2n^2 (U-l)}{k-n(U-l)}}
$$
and especially
$$
\mathfrak{D}(N) \leq \frac{K}{N^{\alpha}}
$$
where $\alpha=\sum_{l=1}^{U-1} \frac{2n^2 (U-l)}{k-n(U-l)}$.
Especially if $k=Un$ we have
$$
\mathfrak{D}(N_1, \dots, N_U) \leq K \prod_{l=1}^{U-1} N_{l}^{-\frac{2n (U-l)}{l}}
$$
and
$$
\mathfrak{D}(N) \leq \frac{K}{N^{\beta}}
$$
where $\beta=\sum_{l=1}^{U-1} \frac{2n (U-l)}{l}$.
\end{thm}
\begin{proof}
Let us use the notation $C_l=(\mathbf{c}_{l,1}^{\top} , \dots , \mathbf{c}_{l,n}^{\top})^{\top}$ for $l=1, \dots, U$.

Let us first fix some small $C_U \in \mathbf{L}_{U}(N_U)$. Now $|C_U|=\OO(1)$. Then write $W_U = \{ (\mathbf{x}_{1}^{\top} , \dots , \mathbf{x}_{n}^{\top})^{\top} | \mathbf{x}_i \in L(\mathbf{c}_{U,1} , \dots , \mathbf{c}_{U,n}) \}$. Then let $V_U = W_{U}^{\bot}$ be its orthogonal complement and $\pi_{U}: \mathcal{M}_{n \times k}(\C) \rightarrow V_U$ an orthogonal projection.

A subspace $V_U$ has $\dim_{\R}(V_U)=2nk-\dim_{\R}(W_U)=2nk-2n^2=2n(k-n)$ so the image $\pi_{U}(\mathbf{L}_{U-1}(N_{U-1}))$ falls into a $2n(n-k)$-dimensional hypercube with side length smaller or equal than $(2nk)^{2} N_{U-1}=\OO(N_{U-1})$ by lemma \ref{projection} with coordinates having restricted length since projection can only shrink. We also have $|\mathbf{L}_{U-1}(N_{U-1})|=\theta(N_{U-1}^{2nk})$ so using the linearity of $\pi_{U}$ and pigeon hole principle we have some $C_{U-1} \in \mathbf{L}_{U-1}(N_{U-1})$ such that
$$
\pi_{U}(C_{U-1})=\OO(\sqrt[2n(k-n)]{\frac{N_{U-1}^{2n(k-n)}}{N_{U-1}^{2nk}}})=\OO(N_{U-1}^{-\frac{n}{k-n}}).
$$

Now similarly build $V_{U-l}=W_{U-l}^{\bot}$ for $l=0, \dots, U-2$ by setting $W_{U-l} = \{ (\mathbf{x}_{1}^{\top} , \dots , \mathbf{x}_{n}^{\top})^{\top} | \mathbf{x}_i \in L(\mathbf{c}_{U,1} , \dots , \mathbf{c}_{U,n} , \mathbf{c}_{U-1,1} , \dots , \mathbf{c}_{U-1,n} , \dots , \mathbf{c}_{U-l,1} , \dots , \mathbf{c}_{U-l,n}) \}$. This gives $\dim_{\R}(V_{U-l})=2nk-\dim_{\R}(W_{U-l})=2nk-2n^2(l+1)=2n(k-nl-n)$. And again we find $C_{U-l-1}$ such that
$$
\pi_{U-l}(C_{U-l-1})=\OO(\sqrt[2n(k-nl-n)]{\frac{N_{U-l-1}^{2n(k-nl-n)}}{N_{U-l-1}^{2nk}}})=\OO(N_{U-l-1}^{-\frac{nl+n}{k-nl-n}}).
$$

Lemma \ref{matriisitulo} gives that if $A=(C_{1}^{\top} , \dots , C_{U}^{\top})^{\top}$ and $B=(\pi_{2}(C_{1})^{\top} , \dots , \pi_{U}(C_{U-1})^{\top}, C_{U}^{\top})^{\top}$ then $\det(AA^{\dag})=\det(BB^{\dag})$ that is of size
$$
\OO((\prod_{l=0}^{U-2} N_{U-l-1}^{-\frac{nl+n}{k-nl-n}} )^{2n}) = \OO(\prod_{l=1}^{U-1} N_{l}^{-\frac{2n^2 (U-l)}{k-n(U-l)}} ).
$$

\end{proof}


\end{document}